\documentclass[10pt,reqno]{amsproc}
\pdfoutput=1

\usepackage{mathpazo} %
\usepackage[scaled]{helvet} %
\usepackage{courier} %
\normalfont
\usepackage[T1]{fontenc}

\usepackage[all]{xy}

\usepackage{graphicx}

\usepackage{empheq}

\usepackage{hyperref}

\newcommand{\RR}{\ensuremath{\mathbf{R}}}
\newcommand{\PP}{\ensuremath{\mathbf{P}}}

\newcommand{\pt}{\ensuremath{\mathrm{pt}}}
\newcommand{\w}{\ensuremath{\mathrm{w}}}
\newcommand{\krv}{\ensuremath{\kappa}}
\newcommand{\into}{\ensuremath{\hookrightarrow}}

\newcommand{\II}{\ensuremath{\mathrm{II}}}

\newcommand{\R}{\ensuremath{R}}
\newcommand{\Rep}{\ensuremath{R^\epsilon}}

\newcommand{\gep}[1]{\ensuremath{g_\epsilon\!\left( #1 \right)}}
\newcommand{\HH}{\ensuremath{\mathcal{H}}}
\newcommand{\VV}{\ensuremath{\mathcal{V}}}

\theoremstyle{plain}
\newtheorem*{theorem*}{Theorem}
\newtheorem{theorem}{Theorem}
\newtheorem{proposition}[theorem]{Proposition}
\newtheorem*{proposition*}{Proposition}
\newtheorem{corollary}[theorem]{Corollary}

\newtheorem{lemma}[theorem]{Lemma}
\newtheorem*{lemma*}{Lemma}

\newtheorem*{corollary*}{Corollary}
\newtheorem*{question*}{Question}
\newtheorem*{idea*}{Idea}

\theoremstyle{definition}
\newtheorem{definition}[theorem]{Definition}
\newtheorem{example}[theorem]{Example}
\newtheorem*{definition*}{Definition}
\newtheorem*{example*}{Example}
\newtheorem*{remark*}{Remark}

\theoremstyle{plain}

\newcommand{\mrm}[1]{\mathrm{#1}}
\newcommand{\mc}[1]{\mathcal{#1}}

\newcommand{\dchi}{\ensuremath{\mrm{d}\chi}}
\newcommand{\dk}{\ensuremath{\mrm{d}\kappa}}

\title{A New Approach to Euler Calculus for Continuous Integrands}
\author{Carl McTague}
\email{carl.mctague@rochester.edu}
\urladdr{\href{http://www.mctague.org/carl}{www.mctague.org/carl}}
\address{Mathematics Department, University of Rochester, Rochester, NY 14627, USA}
\address{Mathematics Department, Johns Hopkins University, Baltimore, MD 21218, USA}
\address{Mathematics Department, University of Southampton, Southampton, SO17 1BJ, UK}

\begin{document}

\begin{abstract}
  Euler calculus is based on integrating simple functions with respect to the Euler characteristic. This paper makes the case for extending Euler calculus to continuous functions by integrating with respect to (Gaussian) curvature. This requires a metric but is nevertheless defined within any O-minimal theory. It satisfies a Fubini theorem and extends to a functor. Euler calculus is the ``adiabatic limit'' of this ``curvature calculus''. All this suggests new applications of differential geometry to data analysis.
\end{abstract}

\maketitle

\section{Overview}

The aim of Euler calculus is to perform integration using the Euler characteristic $\chi$ as measure. This is a reasonable idea since:
\begin{align*}
  \chi(X \cup Y) = \chi(X) + \chi(Y) - \chi(X \cap Y)
\end{align*}
(Provided one uses compactly supported cohomology, that is. For example, if $X\cup Y=\mrm{S}^1$ and $Y$ is a point then, since $\chi(Y)=1$ and $\chi(X\cup Y)=0$, then the formula can hold only if $\chi(X)=-1$.)

But this is only true for \emph{finite} unions so $\chi$ isn't a true measure. As a result the Euler integral of a \emph{simple function} (in the sense of measure theory) is easy to define:
    \begin{align*}
      \int \Big( \sum_\text{finite} \lambda_i \, 1_{V_i} \Big) \dchi
      = \sum_\text{finite} \lambda_i\,\chi(V_i)
    \end{align*}
but behaves poorly under limits:
\begin{align*}
  \lim s_n = \lim s'_n
  \quad\text{doesn't necessarily imply that}\quad
  \lim \int s_n \,\dchi = \lim \int s'_n \,\dchi
\end{align*}
even if the convergence is uniform.

Baryshnikov--Ghrist \cite{baryshnikov-ghrist-2010} studied this failure of convergence. They considered the Euler integrals of two sequences of simple functions converging (uniformly) to a given continuous function~$\alpha$:
\begin{align*}
  \int \alpha \lfloor \dchi \rfloor &= \lim_{n\to\infty} \int \tfrac1n \lfloor n\alpha \rfloor \,\dchi &
  \int \alpha \lceil \dchi \rceil &= \lim_{n\to\infty} \int \tfrac1n \lceil n\alpha \rceil \,\dchi
\end{align*}

These integrals differ in general.

\smallskip
\noindent
\begin{minipage}{\linewidth}
\begin{example}
  \mbox{}
  \begin{center}
    \includegraphics[width=130pt]{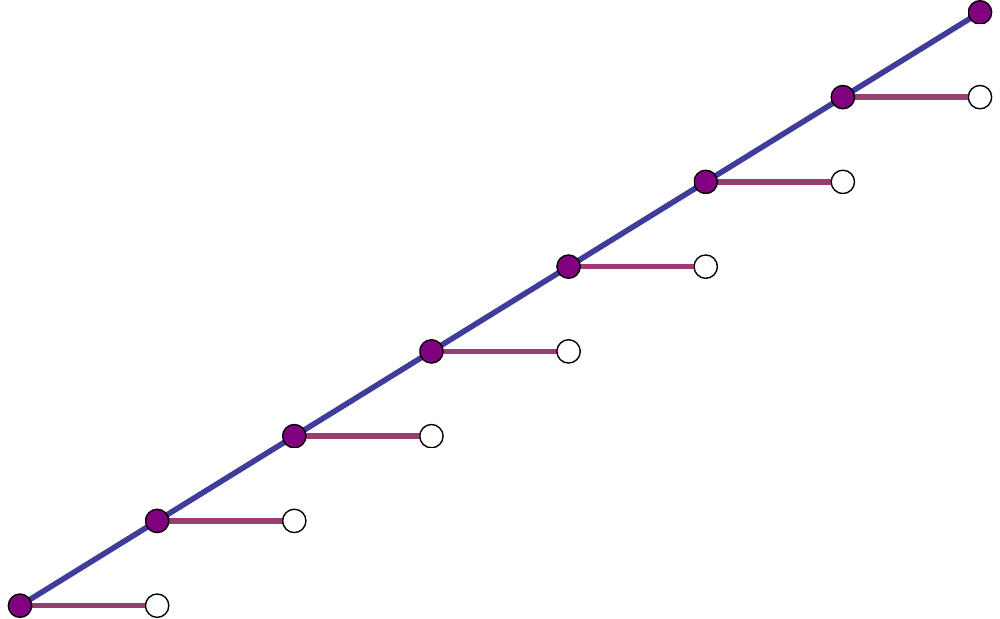}
    \hspace{20pt}
    \includegraphics[width=130pt]{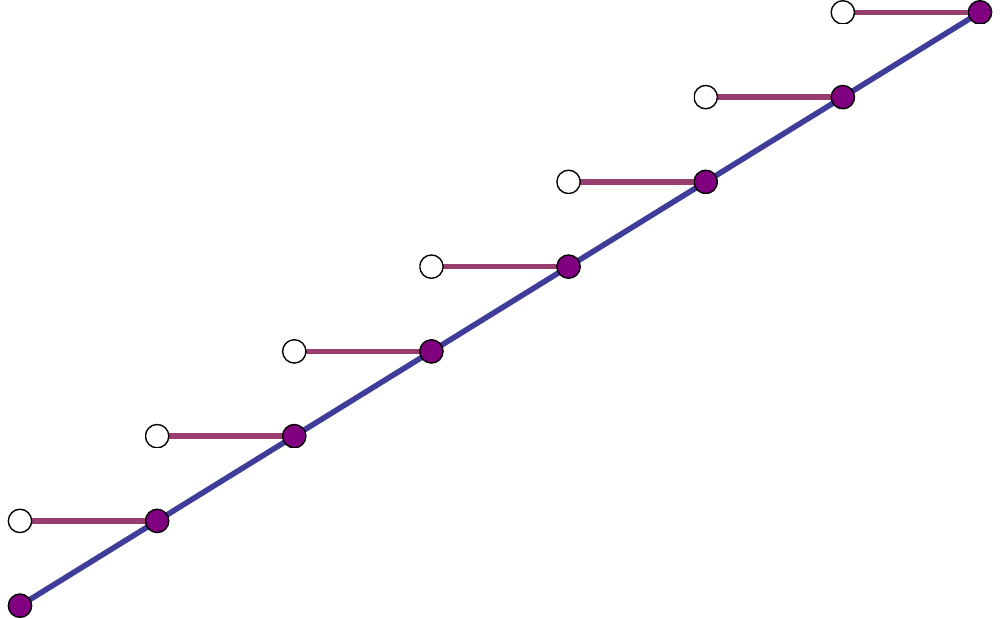} \\
    \nopagebreak[4]
    \hspace{15pt} $\int \mrm{id}_{[0,1]}\;\lfloor\dchi\rfloor=1$
    \hspace{60pt} $\int \mrm{id}_{[0,1]}\;\lceil\dchi\rceil=0$
  \end{center}
\end{example}
\end{minipage}
\smallskip

More generally:

\begin{lemma*}[{\cite[Lemma~1]{baryshnikov-ghrist-2010}}]
  If $\alpha:\Delta \to \RR$ is affine then:
  \begin{align*}
    \int_{\mathring{\Delta}} \alpha \lfloor \dchi \rfloor &= (-1)^{\dim(\Delta)} \inf \alpha &
    \int_{\mathring{\Delta}} \alpha \lceil \dchi \rceil &= (-1)^{\dim(\Delta)} \sup \alpha
  \end{align*}
  where $\mathring{\Delta}$ is the interior of $\Delta$.
\end{lemma*}

Since $\inf$ and $\sup$ are not additive, neither of these integrals is.

\section*{A Na\"ive Starting Point}

For a fresh perspective, consider the problem in the simplicial context. So, for the time being:
\begin{enumerate}
\item A space is a simplicial complex $X$.
\item A \emph{simple function} on $X$ is an $\RR$-linear combination of its simplices.
\item A \emph{continuous function} on $X$ is a simplicial map $\alpha:X \to \RR$, i.e.\ a function defined by assigning a real number to each vertex and extending linearly to the interior of each simplex.
\end{enumerate}
We know how to integrate the former. Our goal is to integrate the latter.

In an intuitive sense easily made precise, there is a unique simple function which best approximates a continuous function $\alpha$, namely:
\begin{align*}
  \sum_{\Delta\in X} \alpha(\hat{\Delta}) \cdot 1_{\mathring{\Delta}}
\end{align*}
where the sum runs over each simplex $\Delta$ of $X$, and where $\mathring{\Delta}$ and $\hat{\Delta}$ denote its interior and barycenter. Since $\chi(\mathring{\Delta})=(-1)^{\dim(\Delta)}$, this simple function has Euler integral:
\begin{align*}
  \int \Big( \sum_{\Delta \in X} \alpha(\hat{\Delta}) \cdot 1_{\mathring{\Delta}} \Big) \dchi =
  \sum_{\Delta \in X} (-1)^{\dim(\Delta)} \alpha(\hat{\Delta})
\end{align*}

So it seems natural to wonder whether a good theory might result if one defined:

\theoremstyle{definition}
\newtheorem*{tent-def}{Tentative Definition}
\begin{tent-def}
  For $X$ and $\alpha : X \to \RR$ \emph{simplicial} let:
\begin{align*}
  \int_X \alpha \, \dchi = \sum_{\Delta \in X} (-1)^{\dim(\Delta)} \alpha(\hat{\Delta})
\end{align*}
where the sum runs over each simplex of $X$.
\end{tent-def}

At least this integral would be additive.

\smallskip
\noindent
\begin{minipage}{\linewidth}
\begin{example}
  Regarding $\mrm{id}_{[0,1]}$ as a simplicial map $\Delta\to\RR$, the tentative definition says:
  \begin{center}
      \includegraphics[width=130pt]{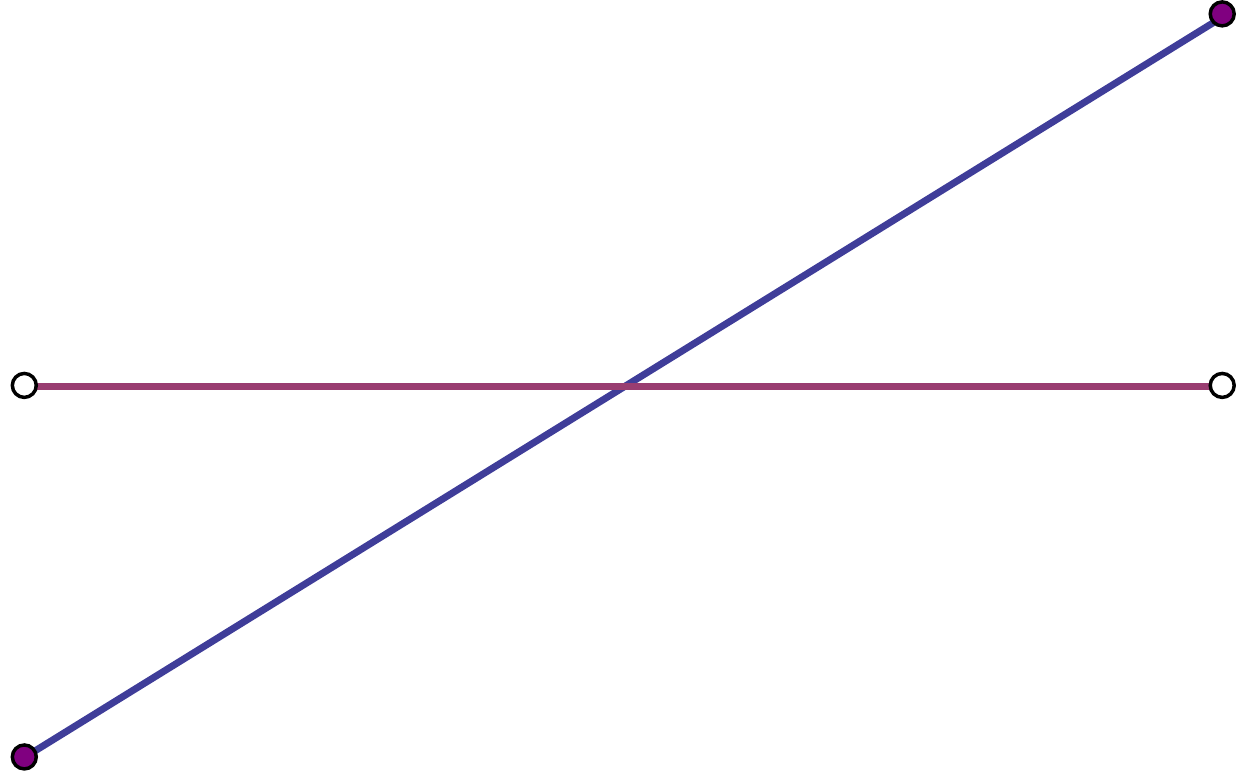} \\
      $\int \mrm{id}_{[0,1]}\,\dchi = \int \big( 0 \cdot 1_{\{0\}} + \tfrac12\cdot1_{(0,1)} + 1\cdot1_{\{1\}}\big)\dchi = 0-\tfrac12+1=\tfrac12$
  \end{center}
\end{example}
More generally, for any simplicial map $\alpha:\Delta\to\RR$, the tentative definition says that:
\begin{align*}
\int \alpha\,\dchi=\alpha(\hat{\Delta})
\end{align*}
\end{minipage}
\smallskip

One could try to extend the tentative definition to a continuous integrand on a topological space by considering a sequence of increasingly accurate simplicial approximations to it. To do so, one would want to know that the definition is stable under subdivision. However, this is \emph{not} true:

\smallskip
\noindent
\begin{minipage}{\linewidth}
\begin{example}
  \mbox{} \\
  \begin{center} 
  \includegraphics[height=95pt]{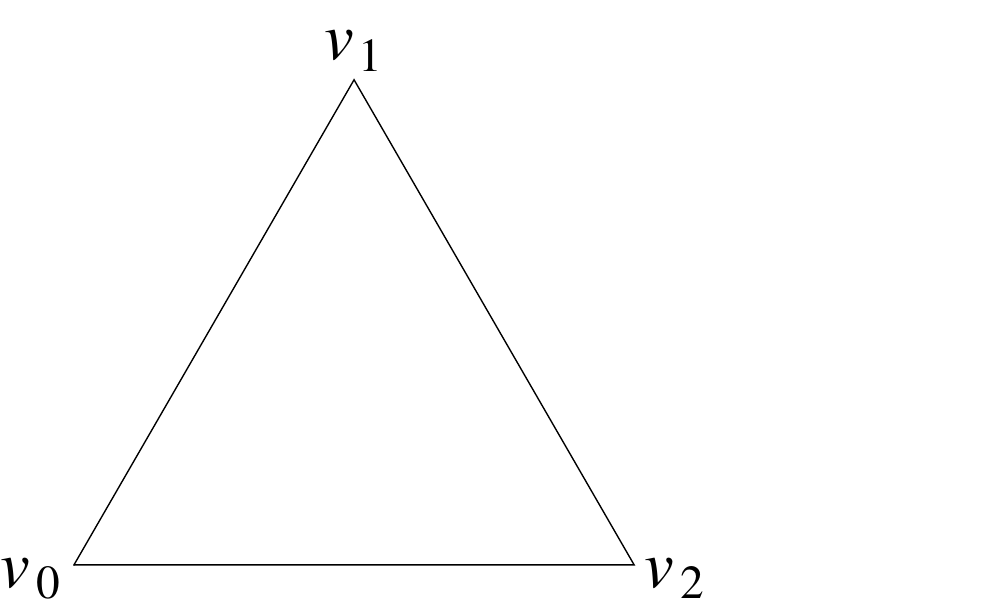}
  {\includegraphics[height=95pt]{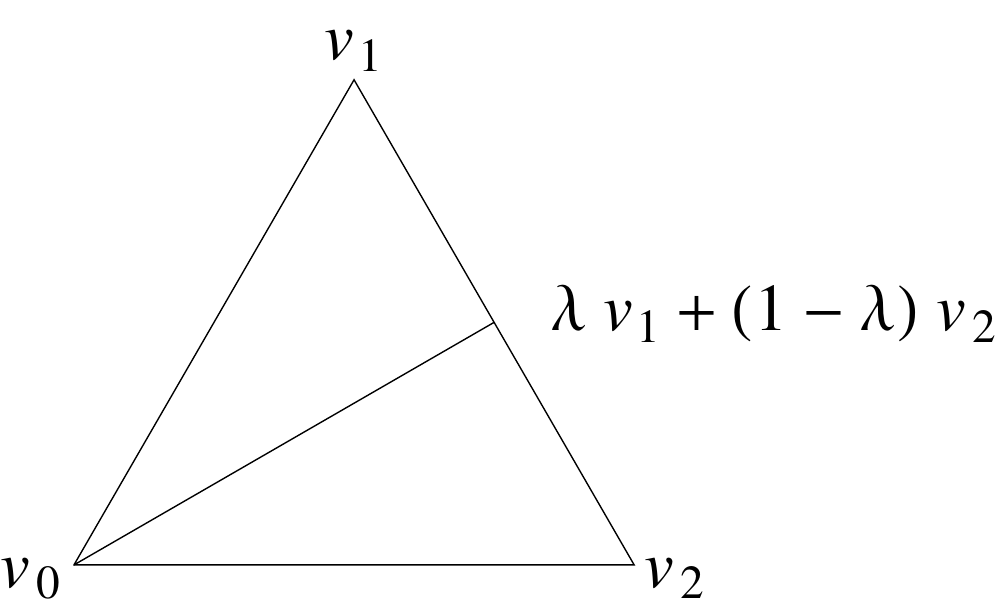}} \\
\hspace{0.0cm} $\int \alpha\,\dchi =\alpha(\hat{\Delta})=\tfrac13\sum\alpha(v_i)$
\hfill {$\int \alpha\,\dchi=\tfrac16\alpha(v_0) + \big(\tfrac13+\tfrac16\lambda\big)\;\alpha(v_1) + \big(\tfrac12-\tfrac16\lambda\big)\;\alpha(v_2)$}
\end{center}
\medskip
These integrals differ for any $0<\lambda<1$. But if one carries out a full barycentric subdivision then, after considerable calculation, one recovers the original integral:
\begin{center}
    \includegraphics[height=100pt]{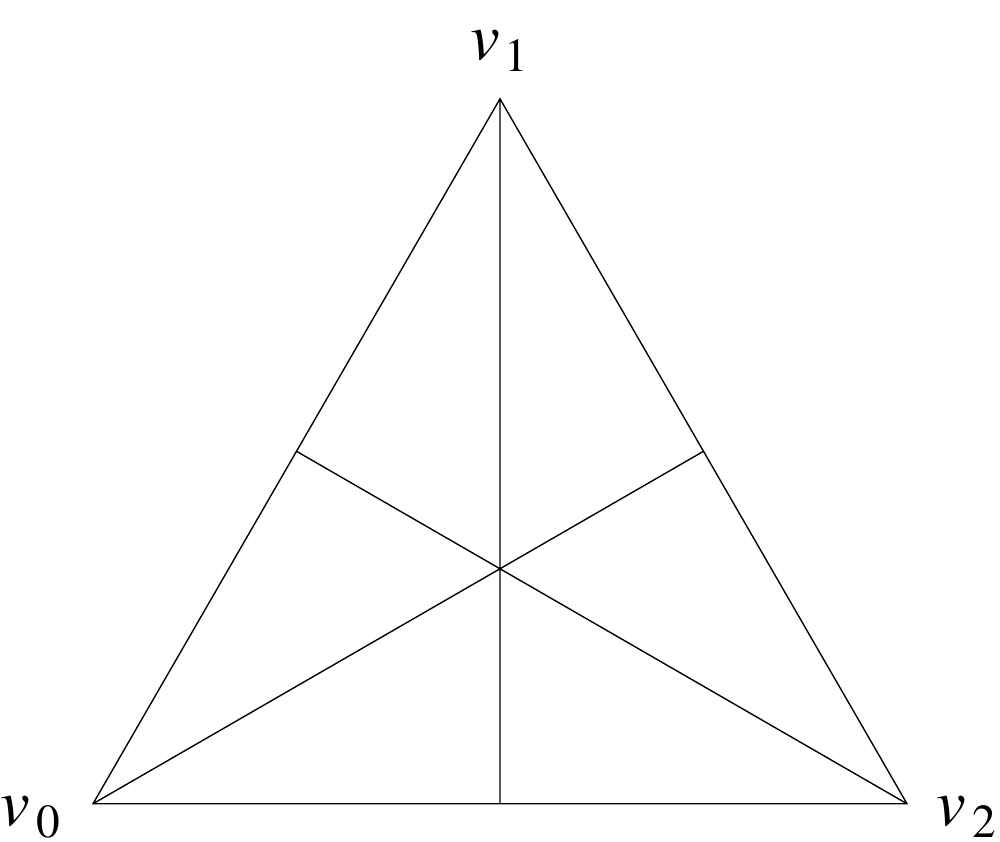} \\
    $\int_{\Delta^{(1)}} \alpha^{(1)}\,\dchi = \alpha(\hat{\Delta})=\int_\Delta \alpha\,\dchi$  
\end{center}
\end{example}
\end{minipage}
\smallskip

This is a general phenomenon:
\begin{theorem}
  \label{thm:crumple}
  For any $n\ge1$:
  \begin{align*}
    \int_X \alpha \; \dchi = \int_{X^{(n)}} \alpha^{(n)} \; \dchi
  \end{align*}
  where $\alpha^{(n)} : X^{(n)}\to\RR$ is the linear extension of $\alpha$ to the $n^\mathrm{th}$ barycentric subdivision $X^{(n)}$ of $X$.
\end{theorem}

The proof is relegated to Appendix~\ref{apx:bary} since the noninvariance of the integral under subdivision turns out to have an entirely different explanation.

\section{Rewriting the Sum}

The tentative definition may be rewritten:
  \begin{align*}
    \int_X \alpha\,\dchi &= \sum_{\Delta \in X} (-1)^{\dim(\Delta)} \alpha(\hat{\Delta}) = \sum_v \alpha(v) \w(v)
  \end{align*}
where $v$ ranges over each \emph{vertex} of $X$ and where:
\begin{align*}
  \w(v) = \sum_i (-1)^i \frac1{i+1} \, \# \big\{ \text{$i$-dimensional simplices containing $v$} \big\}
\end{align*}
This number has a \emph{geometric interpretation.}

  \begin{definition*}[Banchoff \cite{banchoff-1967}]
    Let $X$ be a simplicial complex \emph{embedded in $\RR^N$} (and hence inheriting a metric from it).
    The \emph{curvature} of $X$ at a vertex $v$ is:
  \begin{align*}
    \kappa(v) = \sum_{\Delta\in X} (-1)^{\dim(\Delta)} \mathcal{E}(\Delta,v)
  \end{align*}
  where the \emph{excess angle} $\mathcal{E}(\Delta,v)$ at $v$ of a simplex $\Delta \subset \RR^N$ is:
  \begin{align*}
    \mathcal{E}(\Delta,v) = \frac1{\mrm{vol}(\mrm{S}^{N-1})} \int_{\mrm{S}^{N-1}} \Big[\langle \xi ,v\rangle \ge \langle\xi, x\rangle \text{for all $x$ in $\Delta$}\Big] \mrm{d}\xi
  \end{align*}
  where $\xi$ ranges over the unit sphere $\mrm{S}^{N-1} \subset \RR^N$, and $[P]=
  \begin{cases}
    1 & \text{if $P$} \\ 0 & \text{if $\neg P$}
  \end{cases}
$ is the Iverson bracket.
  \end{definition*}

\theoremstyle{plain}
\newtheorem*{thm-egrigium}{Simplicial Theorem Egrigium}
\begin{thm-egrigium}[{\cite[Thm.~3]{banchoff-1967}}]
  $\kappa(v)$ is intrinsic.
\end{thm-egrigium}

\begin{definition*}
  Given a simplicial complex $X$, let $\mrm{d}_X$ be the unique intrinsic metric which makes each simplex flat and gives each 1-simplex length 1.
Equivalently, let $\mrm{d}_X$ be the intrinsic metric induced by the embedding $X \into \RR^{\{ \text{vertices of $X$} \}}$ carrying a simplex $\Delta$ of $X$ to the convex hull of the vectors $\big\{\tfrac12{\sqrt2}\cdot v \,|\, v \in \Delta\big\}$ where we regard $v$ as a coordinate vector.
\end{definition*}

\begin{theorem}
  \label{thm:weight=curvature}
  For a vertex $v$ of a simplicial complex $X$, $\w(v)$ is equal to the curvature $\krv(v)$ of $X$ at $v$ with respect to the metric $\mrm{d}_X$.
\end{theorem}

  \begin{proof}
    For the metric $\mrm{d}_X$, the excess angle $\mc{E}(\Delta,v)$ is the same at each vertex $v$ of a simplex $\Delta$. Since $\sum_v \mc{E}(\Delta,v)=\chi(\Delta)=1$, it follows that $\mc{E}(\Delta,v)=1/(\dim(\Delta)+1)$. The result follows immediately.
  \end{proof}

\emph{So the integral we are after depends on the metric of the domain, not just its topology.}
This explains why the integral isn't invariant under subdivision:

\medskip
\noindent
\begin{minipage}{\linewidth}
  \begin{example}
    \mbox{}
    \begin{center}
      \includegraphics[width=110pt]{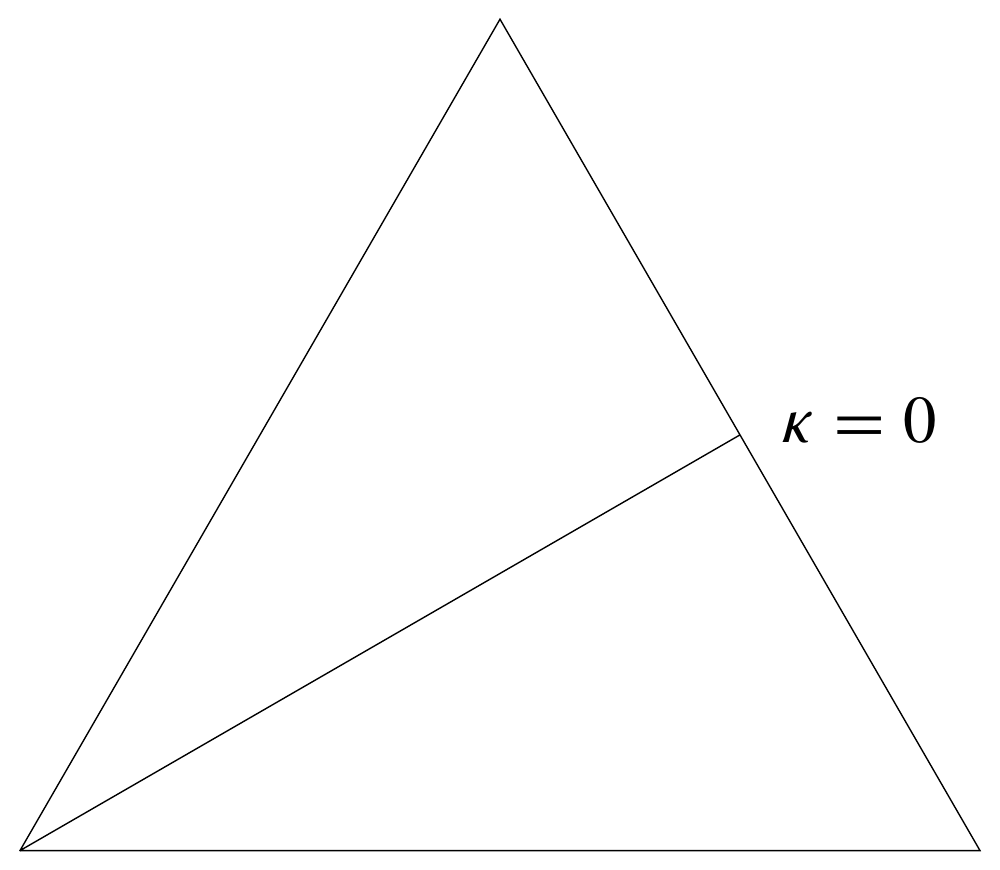}
      \hspace{50pt}
      \includegraphics[width=100pt]{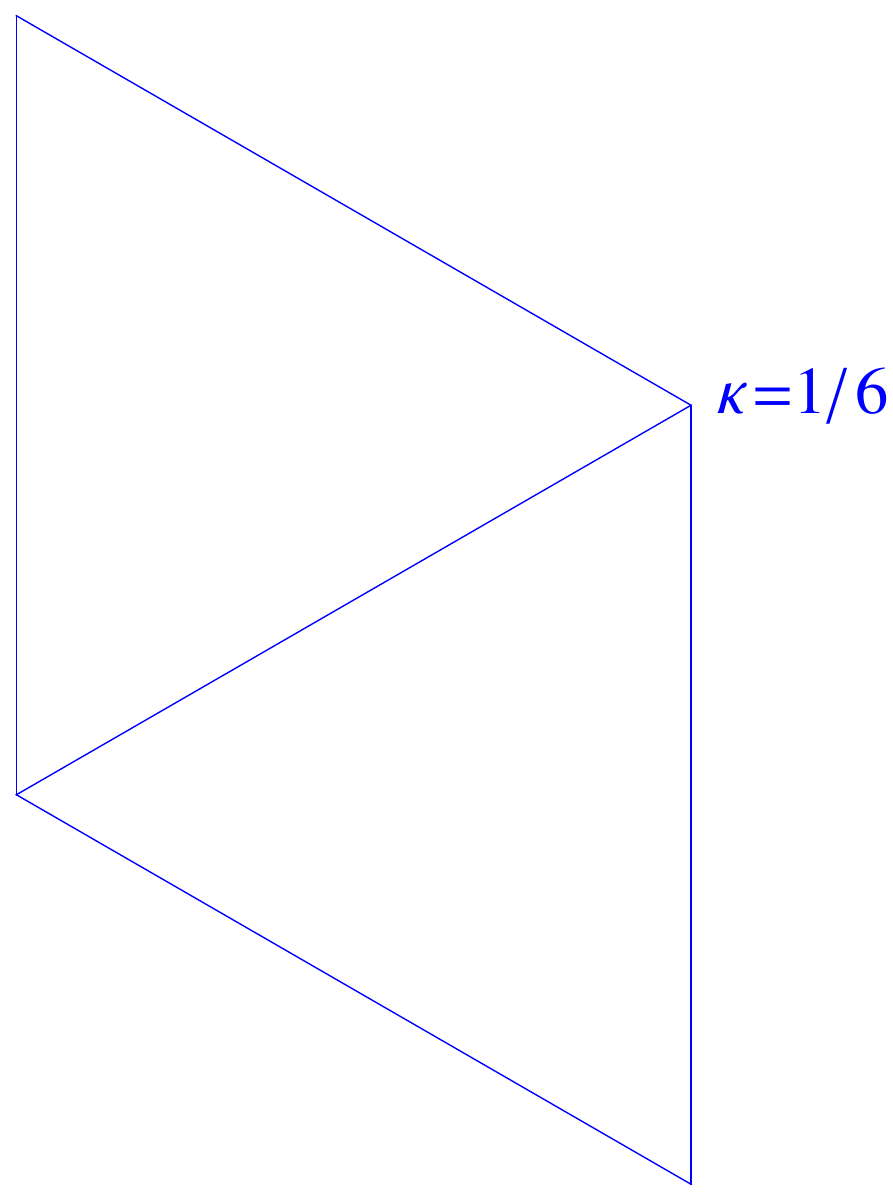} \\
      {We should have integrated like this} \hspace{40pt}
      {but integrated like this instead.}
    \end{center}
  \end{example}
\end{minipage}

\section{Improved Definition}

\theoremstyle{definition}
\newtheorem*{improved-def}{Improved Definition}
\begin{improved-def}
  For a simplicial complex $X$ embedded in $\RR^N$ and a simplicial map $\alpha:X\to\RR$, let:
    \begin{align*}
      \int_X \alpha\,\dchi = \sum_v \alpha(v) \kappa(v)
    \end{align*}
\end{improved-def}

That is, \emph{let Euler integration be integration with respect to curvature.} This makes a lot of sense in hindsight since the generalized Gauss-Bonnet theorem \cite{chern-1944} may be interpreted as saying that \emph{curvature is infinitesimal Euler characteristic}:

\newcommand{\Pf}{\ensuremath{\mathrm{Pf}}}

\theoremstyle{plain}
\newtheorem*{cgb-thm}{Generalized Gauss-Bonnet Theorem}
\begin{cgb-thm}
For a compact Riemannian manifold~$M^n$:
  \begin{align*}
    \chi(M) = \int_M (2\pi)^{-n} \Pf(\Omega)
  \end{align*}
  where $\Pf(\Omega)$ is the Pfaffian of the curvature 2-form $\Omega$ of $M$.

  More generally, for a compact Riemannian manifold~$M$ with boundary $\partial M$:
  \begin{align*}
    \chi(M) = \int_M (2\pi)^{-n} \Pf(\Omega) + \int_{\partial M} (2\pi)^{-n}H
  \end{align*}
  where $H$ is the \emph{``geodesic curvature''} of the boundary.
\end{cgb-thm}

\section*{The Importance of the Boundary Contribution}

Since the generalized Gauss-Bonnet theorem only applies to compact spaces, \textit{one should only perform curvature integration over compact domains.} So to compute the curvature integral of a function one should write it as a linear combination of continuous functions over compact domains. %

\begin{example}
  An open interval $X=(0,1)$ has curvature 0 yet has $\chi(X)=-1$. But if we write:
  \begin{align*}
    \int 1_{(0,1)} \,\dchi
    &= \int \big( 1_{[0,1]} - 1_{\{0\}} - 1_{\{1\}} \big) \,\dchi
  \end{align*}
  then curvature integration correctly computes:
  \begin{align*}
    \int_{[0,1]} 1\,\dk_{[0,1]} - \int_{\{0\}} 1\,\dk_{\{0\}} - \int_{\{1\}}1\,\dk_{\{1\}} &= \big(\tfrac12+\tfrac12\big)-1-1 = -1
  \end{align*}
  (The curvature of $[0,1]$ is concentrated at the endpoints, each having measure $1/2$.)
\end{example}
  
\section{Curvature Is as General as Euler Characteristic}

The Euler characteristic can be defined in great generality---for any space within any O-minimal theory \cite{vanDenDries-1998}. Curvature can be defined in the same generality. Indeed, Br\"ocker-Kuppe \cite{broecker-kuppe-2000} used Goresky-MacPherson's work \cite{smt88} on stratified Morse theory to define the curvature \emph{of any tamely stratified space.} This includes any space within any O-minimal theory. We briefly summarize their work.

Loosely speaking, a Morse function $f:Y\to\RR$ on a Whitney stratified space $Y$ is one which restricts to a classical Morse function on each stratum. (For the complete definition, see \cite[\S I.2.1]{smt88}.) The topological change in $Y$ near a critical point $y$ of $f$ with value $v=f(y)$ is described by the \emph{local Morse data:}
  \begin{align*}
    (P,Q) = \mrm{B}(y,\delta) \cap \Big(
    f^{-1}[v-\epsilon,v+\epsilon], \,
    f^{-1}[v-\epsilon]
    \Big)
  \end{align*}
  where $\mrm{B}(y,\delta)$ is a closed ball of radius $\delta$ centered at $y$. Its topological type is independent of $\epsilon$ and $\delta$ provided $0<\epsilon\ll \delta \ll 1$, and its Euler characteristic:
  \begin{align*}
    \chi(P,Q)=\chi(P)-\chi(Q)=1-\chi(Q)
  \end{align*}
  is called the \emph{index} of $f$ at $y$ and denoted $\alpha(f,y)$. If $y$ is not critical then the same construction gives $\alpha(f,y)=0$. If $Y$ is compact then:
  \begin{align}
    \chi(Y)=\sum_{y \in Y} \alpha(f,y)
    \label{alphasum}
  \end{align}

To define the curvature of a Whitney stratified space $Y \subset \RR^N$, Br\"ocker-Kuppe consider all possible projections $\mrm{h}_x(y)=-\langle x,y\rangle$ as $x$ ranges over the unit sphere $\mrm{S}^{N-1}\subset\RR^N$. If the stratification of $Y$ is \emph{tame} \cite[Def.~3.4]{broecker-kuppe-2000} then $\mrm{h}_x$ is a Morse function for almost all such $x$. (Every space within any O-minimal theory admits a tame stratification, see \cite[Ex.~3.6.d]{broecker-kuppe-2000}.) The \emph{curvature measure} of a Borel set $U$ is then the average contribution to the sum (\ref{alphasum}) of the points of $U$ as $f$ runs over all Morse projections $\mrm{h}_x$:
\begin{align*}
  \kappa_Y(U) = \frac1{\mrm{vol}(\mrm{S}^{N-1})} \int_{\mrm{S}^{N-1}} \sum_{y \in U} \alpha(\mrm{h}_x,y) \mrm{d}x
\end{align*}

\begin{example}
  A simplicial complex $X$ embedded in $\RR^N$ may be regarded as a tamely stratified space. Its curvature measure $\kappa_X$ is the discrete measure $\sum_{v \in X} \kappa(v) \delta_v$ where $\kappa(v)$ is the curvature defined by Banchoff and $\delta_v(U)=1_U(v)$ is the Dirac measure.
\end{example}

\begin{example}[Br\"ocker-Kuppe, Ex.~4.2.a]
  \mbox{}
  \begin{center}
    \includegraphics[width=160pt]{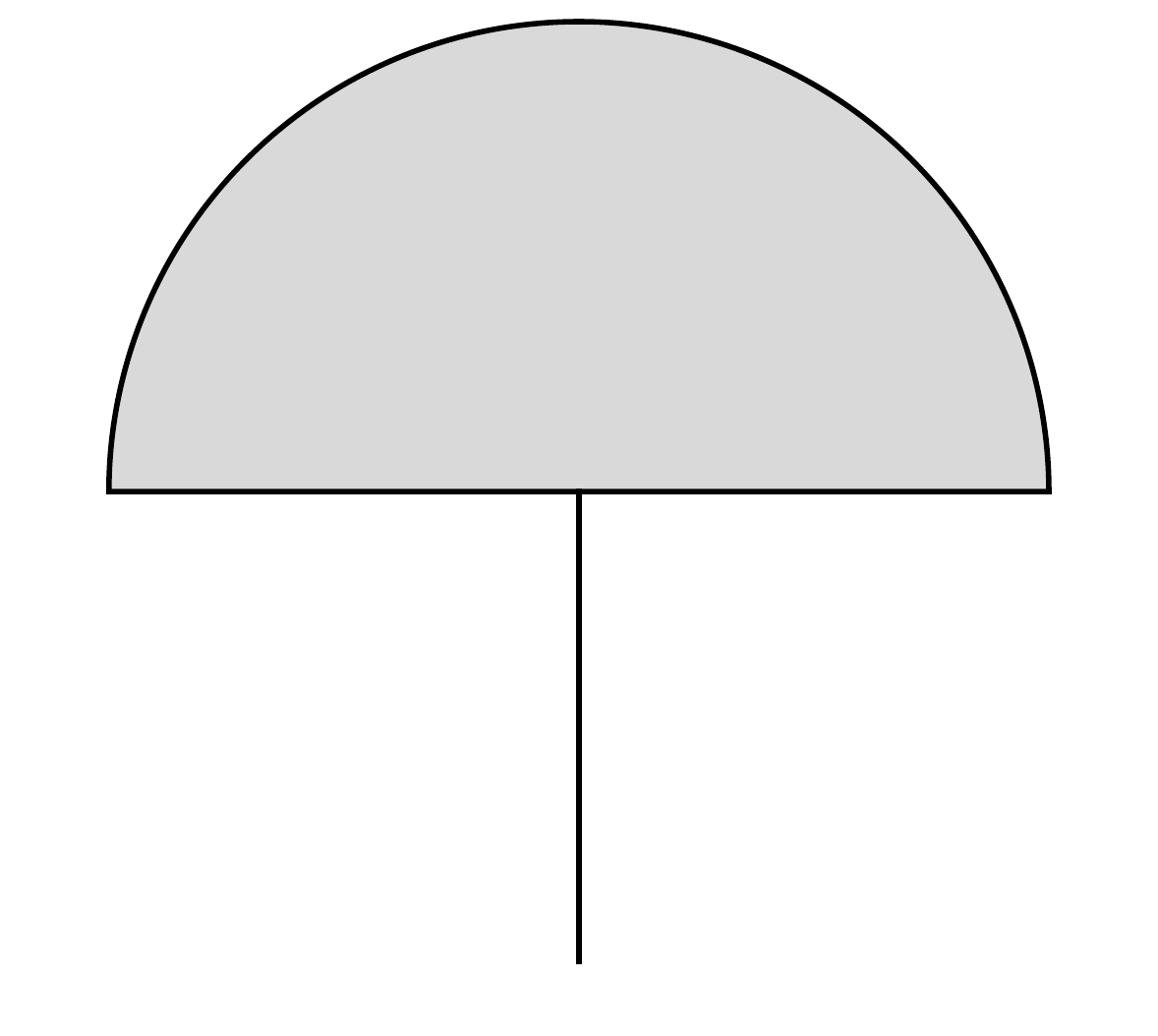}
    \includegraphics[width=160pt]{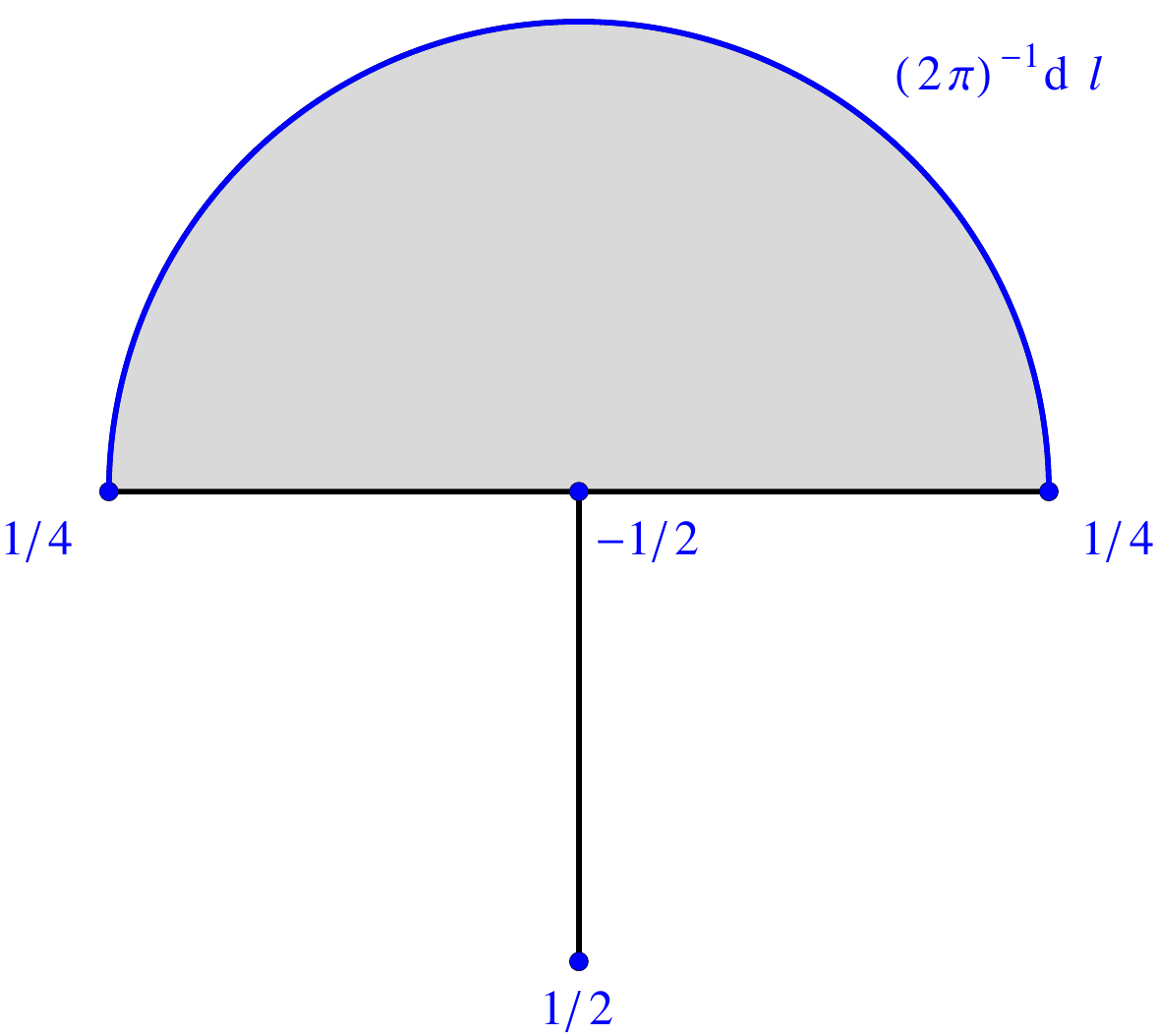}
  \end{center}
\end{example}

As these examples illustrate, the curvature measure $\kappa_Y$ reflects not just the curvature of the smooth strata of $Y$ but also the metric near singularities. This is because the local Morse data is the product of the tangential Morse data and the normal Morse data, which leads to a factorization of the index $\alpha(f,y)=\alpha_{\mrm{tan}}(f,y) \, \alpha_{\mrm{nor}}(f,y)$.

\theoremstyle{plain}
\newtheorem*{stratified-gauss-bonnet}{Stratified Gauss-Bonnet Theorem}
\begin{stratified-gauss-bonnet}[Br\"ocker-Kuppe, Prop.~4.1]
  If $Y \subset \RR^N$ is a compact tamely stratified space then $\kappa_Y(Y)=\chi(Y)$, that is:
  \begin{align*}
    \int 1_Y \, \dk_Y = \int 1_Y \, \dchi
  \end{align*}
\end{stratified-gauss-bonnet}

\theoremstyle{definition}
\newtheorem*{generalized-def}{Generalized Definition}
\begin{generalized-def}
  For a compact tamely stratified space $Y \subset \RR^N$ and a continuous function $\alpha:Y\to\RR$, let:
  \begin{align*}
    \int_Y \alpha\,\dchi = \int_Y \alpha\,\dk_Y
  \end{align*}
  where the right hand side is the Lebesgue integral of $\alpha$ with respect to the Br\"ocker-Kuppe curvature measure $\kappa_Y$.
\end{generalized-def}

The measure $\kappa_Y$ can always be expressed in terms of the canonical measures of the strata of $Y$:

\begin{proposition*}[Br\"ocker-Kuppe, Prop.~4.4]
  For each stratum $S$ of a tamely stratified space $Y \subset \RR^N$ there is a function $\kappa(Y,S):S\to\RR$ such that:
  \begin{align*}
    \kappa_Y(U) = \sum_S \int_{U \cap S} \kappa(Y,S) \mrm{d}S
  \end{align*}
  where $\mrm{d}S$ is the canonical measure on $S$.
\end{proposition*}

So to perform curvature integration one simply performs Riemannian integration along each stratum $S$, after first multiplying the integrand by the corresponding curvature function $\kappa(Y,S)$:
\begin{align*}
  \int_Y \alpha\,\dk_Y = \sum_S \int_S \alpha\cdot\kappa(Y,S)\,\mrm{d}S
\end{align*}

\theoremstyle{definition}
\newtheorem*{final-def}{Final Definition}
\begin{final-def}
  Given a finite collection $\{K_i\}$ of compact tamely stratified subsets of $Y$, and continuous functions $\alpha_i:K_i\to\RR$, consider their sum $\sum_i \alpha_i$ as a (possibly discontinuous) function on $Y$ and define:
  \begin{align*}
    \int_Y \Big(\sum_i \alpha_i\Big)\,\dchi = \sum_i \Big(\int_{K_i} \alpha_i\,\dk_{K_i}\Big)
  \end{align*}
\end{final-def}
\noindent This includes the classical Euler integral of simple functions as a special case.

\section{Fubini and Functoriality of Euler Integration}

The multiplicativity of the Euler characteristic $\chi(Y\times Z)=\chi(Y)\cdot\chi(Z)$ implies:
\theoremstyle{plain}
\newtheorem*{fubini-simple}{Fubini Theorem for Euler Integration}
\begin{fubini-simple}
\begin{align*}
  \int_Y \left( \int_Z s \,\dchi\right)\dchi = \int_{Y\times Z} s\,\dchi = \int_Z \left( \int_Y s \,\dchi\right)\dchi
\end{align*}
\end{fubini-simple}

In fact it implies much more:

\theoremstyle{plain}
\newtheorem*{functoriality-simple}{Functoriality of Euler Integration}
\begin{functoriality-simple}[cf \cite{macpherson74}]
 Integrating a simple function $s$ along the fibers of a map $f:X\to Y$ produces a simple function $f_*(s)$ on $Y$ with the same Euler integral as $s$:
\begin{align*}
  \int_X s\,\dchi = \int_Y\Big( \underbrace{\int_{f^{-1}(z)} s\,\dchi}_{f_*(s)}\Big)\,\dchi
\end{align*}
This pushforward $f_*$ is a homomorphism from the group of simple functions on $X$ to those on $Y$, and the association $f\mapsto f_*$ is \textit{functorial}:
\begin{align*}
  (f\circ g)_*=f_*\circ g_*
\end{align*}
\end{functoriality-simple}

So Euler integration, being the special case $c_*$ where $c:X\to\mrm{pt}$, extends to a functor.

\smallskip

This leads to a divide-and-conquer method for computing the Euler characteristic of a space $X$:
\begin{enumerate}
\item Fiber $X$ over another space $f:X\to Y$.
\item Compute the Euler characteristics of the fibers $f_*(1_X)(x)=\chi\big(f^{-1}(x)\big)$.
\item Amalgamate these numbers through Euler integration $\chi(X)=\int f_*(1_X)\,\dchi$.
\end{enumerate}

\pagebreak[3] %

\begin{example}
    \label{ex:sphere-projection}
    Consider the projection $f:\mrm{S}^2\to [-\pi/2,\pi/2]$ of the unit 2-sphere in $\mathbf{R}^2$ defined in spherical coordinates by $(r,\theta,\phi)\mapsto\theta$.
  \begin{center}
  \includegraphics[width=100pt]{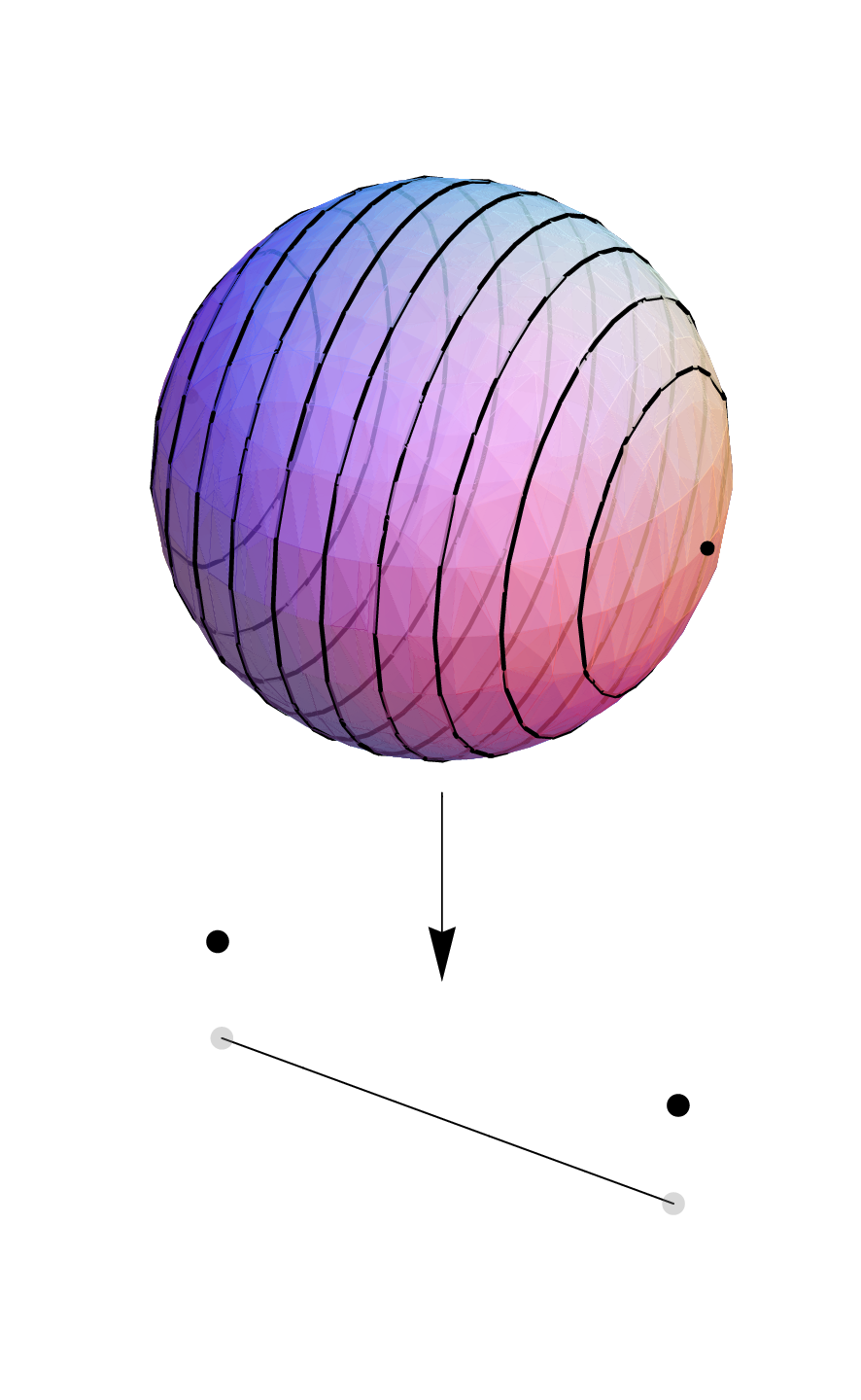}
  \end{center}
  The generic fiber has $\chi(\mathrm{S}^1)=0$, and the two exceptional fibers have $\chi(\pt)=1$. So:
  \begin{align*}
    \chi(\mrm{S}^2)=\int 1_{\mrm{S}^2}\,\dchi=\int f_*(1_{\mrm{S}^2})\,\dchi = \int \big(1_{\{2\pi\}}+1_{\{-2\pi\}}\big)\,\dchi=1+1=2
  \end{align*}
\end{example}

\section{Fubini and Functoriality of Curvature Integration}

Curvature integration also satisfies a Fubini theorem: The curvature measure of a (metric) product is the product of the curvature measures $\kappa_{Y\times Z}=\kappa_Y \times \kappa_Z$.
So the general Fubini theorem (\cite[Thm.~19 of \S12.4]{royden-1988}) applies:
\theoremstyle{plain}
\newtheorem*{fubini-curvature}{Fubini Theorem for Curvature Integration}
\begin{fubini-curvature}
\begin{align*}
  \int_Y \left( \int_Z \alpha\,\mrm{d}\kappa_Z \right) \mrm{d}\kappa_Y = \int_{Y\times Z} \alpha\,\mrm{d}(\kappa_{Y\times Z}) = \int_Z \left( \int_Y \alpha\,\mrm{d}\kappa_Y \right) \mrm{d}\kappa_Z
\end{align*}
\end{fubini-curvature}

Functoriality is more complicated. One can push forward curvature along a map $f:X\to Y$ in the sense of \emph{measure theory}. This is functorial but the resulting measure $f_*(\kappa_X)$ is typically \emph{not absolutely continuous} with respect to $\kappa_Y$. So it is unrealistic to hope for a \emph{function} $f_*(\alpha)$ such that $f_*(\alpha\cdot\kappa_X)=f_*(\alpha)\cdot\kappa_Y$.

\begin{example}
  Reconsider Example~\ref{ex:sphere-projection} from the curvature point of view. The unit 2-sphere in $\mathbf{R}^2$ has constant curvature. Its measure-theoretic pushforward smears into the interior of $[-2\pi,2\pi]$ where there is no curvature. So $f_*(\kappa_{\mrm{S}^2})$ cannot be expressed as $\beta\cdot\kappa_{[-\pi/2,\pi/2]}$ for any function $\beta$.
  \begin{center}
    \includegraphics[width=100pt]{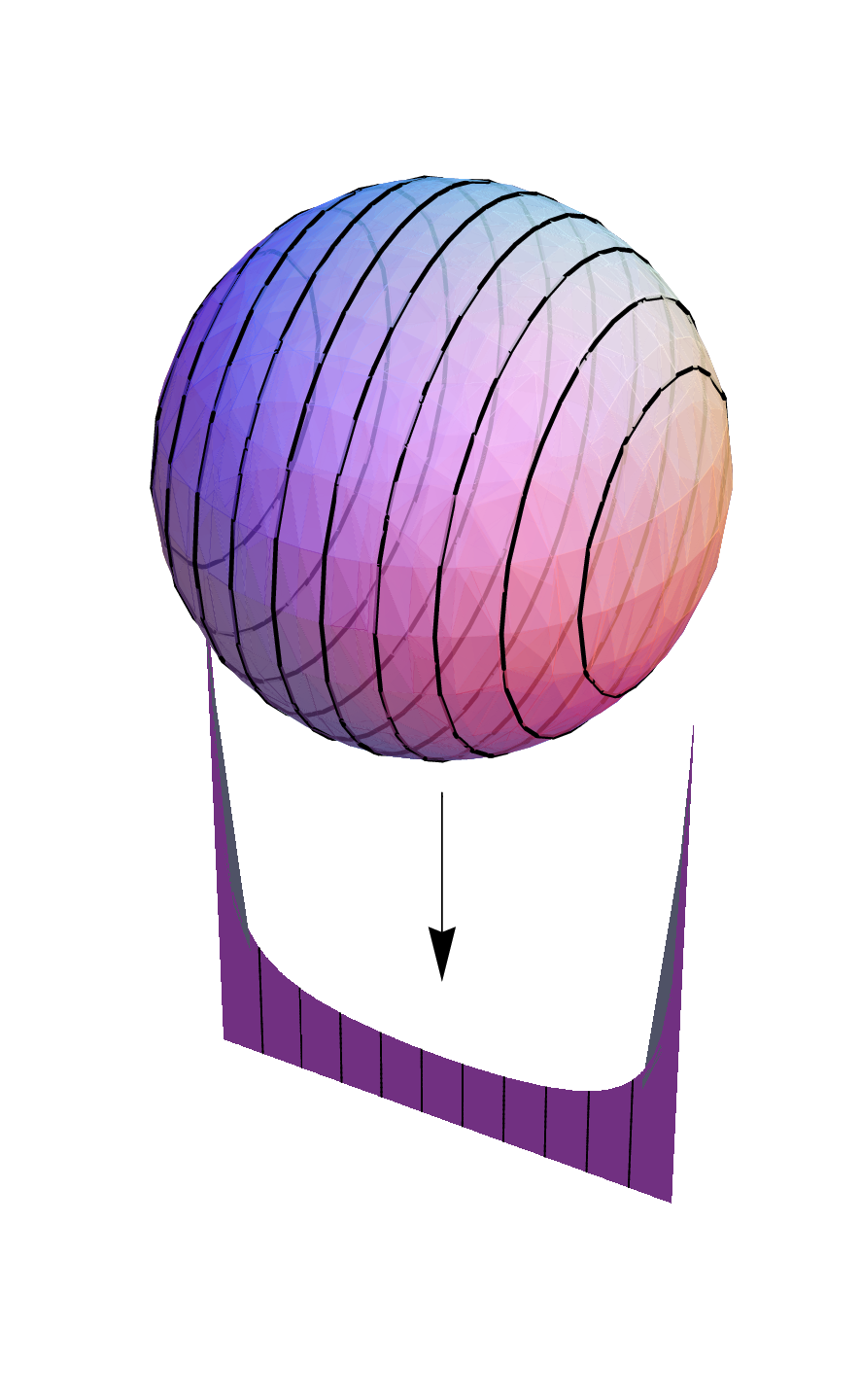}
  \end{center}
\end{example}

The reason absolute continuity fails is that the curvature of $X$ reflects not just the \emph{intrinsic} geometry of the fibers but also their \emph{extrinsic} geometry. Put differently, the curvature of $X$ need not split into vertical and horizontal parts.

\medskip

This nonsplitting of curvature can be described precisely for a \emph{submersion $\pi:(M,g)\to(B,g_B)$ of semi-Riemannian manifolds}, that is, a smooth map whose derivative $\mrm{D}\pi$ restricts to an isometry:
\begin{align*}
  \mrm{D}\pi|_{\mrm{H}M}:\mrm{H}M\to\mrm{T}B
\end{align*}
from the \emph{horizontal bundle} $\mrm{H}M$, i.e.\ the orthogonal complement of the \emph{vertical bundle} $\mrm{V}M=\ker \mrm{D}\pi\subset\mrm{T}M$, to the tangent bundle $\mrm{T}B$ of $B$.

The O'Neill tensors \cite{oneill-1966,gray-1967} are to a submersion what the second fundamental form is to an immersion. Karcher \cite{karcher-1999} formulated them elegantly in terms of the covariant derivatives of the orthogonal projections:
\begin{align*}
  \HH,\VV&:\mrm{T}M\to\mrm{T}M
\end{align*}
onto $\mrm{H}M$ and $\mrm{V}M$.

Let $\R$ be the Riemann tensor of the Levi-Civita connection on $\mrm{T}M$, and let $\R^\mrm{V}$ and $\R^\mrm{H}$ be the Riemann tensors of the induced connections on $\mrm{V}M$ and $\mrm{H}M$.

\theoremstyle{plain}
\newtheorem*{karchers-formula}{Karcher's Formula}
\begin{karchers-formula}[{\cite[Eqns.~(10)--(12)]{karcher-1999}}]
  For $V$ in $\mrm{V}M$, $H$ in $\mrm{H}M$, and $X,Y$ in $\mrm{T}M$:
\begin{align*}
  \R(X,Y)V &= -\overbrace{\R(X,Y)\HH\cdot V}^{\text{in $\mrm{H}M$}} \,+ \overbrace{\R^{\mrm{V}}(X,Y)V\,-\big[\nabla_X\HH,\nabla_Y\HH\big]V}^{\text{in $\mrm{V}M$}} \\
  \R(X,Y)H &= \phantom{-}\underbrace{\R(X,Y)\HH\cdot H}_{\text{in $\mrm{V}M$}} + \underbrace{\R^{\mrm{H}}(X,Y)H-\big[\nabla_X\HH,\nabla_Y\HH\big] H}_{\text{in $\mrm{H}M$}}
\end{align*}
\end{karchers-formula}
\noindent (Beware that $\R(X,Y)\HH\cdot Z=(\nabla^2_{X,Y}-\nabla^2_{Y,X})\HH\cdot Z$ does \emph{not} satisfy a cyclic Bianchi identity.)

This breaks the curvature form $\Omega$ of $M$ into blocks: Let $V_1,\dots,V_k$, $H_{k+1},\dots,H_n$ be a positively oriented orthonormal frame on $M$ consisting of vertical followed by horizontal vectors. In terms of this basis then:
\begin{corollary}
\begin{align*}
  \Omega(X,Y)
  &=
  \left[
    \arraycolsep=4pt\def\arraystretch{1.4}
    \begin{array}{c|c}
      g\big(\R^{\mrm{V}}(X,Y)V_i-[\nabla_X\HH,\nabla_Y\HH] V_i,V_j\big) & g\big(\R(X,Y)\HH \cdot H_i,V_j\big) \\ \hline
      -g\big(\R(X,Y)\HH \cdot V_i,H_j\big) & g\big(\R^{\mrm{H}}(X,Y)H_i-[\nabla_X\HH,\nabla_Y\HH]H_i,H_j\big)
    \end{array}
  \right]
\end{align*}
\end{corollary}
\noindent (If $X,Y$ are vertical then the top left block is the Gauss equation of the fibers.\footnote{Indeed, $\nabla g=0$ implies that $g(\nabla_X\HH\cdot Y,Z)=g(Y,\nabla_X\HH\cdot Z)$ and hence that:
\begin{align*}
  g(\R(X,Y)V_i,V_j) &=
  g(\R^{\mathrm{V}}(X,Y)V_i,V_j) - g(\nabla_X\HH\cdot V_i,\nabla_Y\HH\cdot V_j) + g(\nabla_Y\HH\cdot V_i,\nabla_X\HH\cdot V_j) \\
  &=g(\R^{\mathrm{V}}(X,Y)V_i,V_j) - g(\II(X,V_i),\II(Y,V_j)) + g(\II(Y,V_i),\II(X,V_j))
\end{align*}
})

\begin{corollary}
  The tensor $\R\HH$ is the obstruction to splitting $\Omega$ into vertical and horizontal parts.
\end{corollary}

\begin{corollary}
  If $\nabla\HH=0$ then $\Omega$ splits:
  \begin{align*}
    \Omega &=
    \left[
      \begin{array}{c|c}
        \Omega^\mathrm{V} & 0 \\ \hline
        0 & \Omega^\mathrm{H}
      \end{array}
    \right]
  \end{align*}
  where $\Omega^\mathrm{V}$ and $\Omega^\mathrm{H}$ are the curvature forms of the induced connections on $\mrm{V}M$ and $\mrm{H}M$.
\end{corollary}

\begin{corollary}
  If $\nabla\HH=0$ then $\Pf(\Omega)=\Pf(\Omega^\mathrm{V})\wedge\Pf(\Omega^\mathrm{H})$.
\end{corollary}

According to the generalized Gauss-Bonnet theorem then:

\begin{corollary}
  If $\nabla\HH=0$ and $\pi$ is proper then integrating along the fibers gives:
  \begin{align*}
    \pi_*\big(\Pf(\Omega)\big)&=\chi(\pi)\cdot\Pf(\Omega^B)
  \end{align*}
\end{corollary}

Note that if the fibers of $\pi$ are \emph{totally geodesic} then the second fundamental forms of the fibers:
\begin{align*}
  \II(U,V)=-\nabla_U\HH\cdot V %
\end{align*}
vanish. But this does not necessarily mean that $\nabla\HH=0$ (it may not vanish in the horizontal direction).

\section{Euler Calculus Is the Adiabatic Limit of Curvature Calculus}

Now shrink the fibers of a submersion $\pi:(M,g)\to(B,g_B)$ of semi-Riemannian manifolds by a factor of $(1-\epsilon)$:
\begin{align*}
  \gep{X,Y} &= g\big((1-\epsilon\VV)\cdot X, (1-\epsilon\VV)\cdot Y\big) & 0&\le\epsilon<1
\end{align*}
Let $V^\epsilon_i=\tfrac1{1-\epsilon}V_i$. Then $V^\epsilon_1,\dots,V^\epsilon_k,H_{k+1},\dots,H_n$ is a positively-oriented orthonormal moving frame on $(M,g_\epsilon)$.

\newcommand{\Rdefo}[1]{\ensuremath{20_#1}}
The methods of Karcher's paper (see \cite{mctague-2015}) may be used to show that:
\begin{proposition}
  \label{prop:defo}
  \begin{align*}
      g_\epsilon\big(\Rep(X,Y)V^\epsilon_i,V^\epsilon_j\big)
    &= g\big(\R^{\mrm{V}}(X,Y)V_i,V_j\big)-(1-\epsilon)^2\cdot g\big([\nabla_X\HH,\nabla_Y\HH] V_i,V_j\big) \\
    g_\epsilon\big(\Rep(X,Y)H_i,V^\epsilon_j\big) &=
        (1-\epsilon)\cdot \Big[ g\big(\R(X,Y)\HH\cdot H_i,V_j\big) \\
        &\hspace{6em} + \epsilon(2-\epsilon)\cdot g\big(\nabla_X\HH\cdot\nabla_{H_i}\HH\cdot Y - \nabla_Y\HH\cdot\nabla_{H_i}\HH\cdot X,V_j\big) \Big] \\
  g_\epsilon\big(\Rep(X,Y)H_i,H_j\big) &= \epsilon(2-\epsilon)\cdot g^B\big(R^B\big(\mrm{D}\pi\cdot X,\mrm{D}\pi\cdot Y\big)\big(\mrm{D}\pi\cdot H_i\big),\mrm{D}\pi\cdot H_j\big) \\ &\hspace{6em}+ (1-\epsilon)^2\cdot g\big(R(X,Y)H_i,H_j\big)
  \end{align*}
\end{proposition}

\begin{corollary}
  \begin{align*}
    \Omega^\epsilon
    =
  \left[
    \begin{array}{c|c}
      \Omega^\mathrm{V} + O(1-\epsilon)^2 & O(1-\epsilon) \\ \hline
      O(1-\epsilon) & \pi^*(\Omega^B) + O(1-\epsilon)^2
    \end{array}
  \right]
  \end{align*}
  where $\pi^*(\Omega^B)$ is the pullback of the curvature form of $B$. So as $\epsilon\to1$:
  \begin{align*}
    \Pf(\Omega^\epsilon) &\to \Pf(\Omega^\mathrm{V}) \wedge \pi^*(\Omega^B)
\intertext{If $\pi$ is proper then:}
    \pi_*\big(\Pf(\Omega^\epsilon)\big) &\to \chi(\pi)\cdot\Pf(\Omega^B)
  \end{align*}
\end{corollary}

\begin{definition}
  There is a sensible notion of ``stratified Riemannian submersion''.
\end{definition}

\begin{proposition}
  \label{prop:strat-sub}
  If $f:Y\to B$ is a proper stratified Riemannian submersion then as $\epsilon\to1$:
  \begin{align*}
    f_*(\kappa^\epsilon_Y) \to \chi(f)\cdot\kappa_B
  \end{align*}
\end{proposition}

$\PP^{m+n+1}$ may be described as the \emph{join} $\PP^m\star\PP^n$ via the orbit structure of the action of $\PP^1$ (regarded as commutative monoid) on $\PP^{m+n+1}$ defined in homogeneous coordinates by:
\begin{align*}
  [\epsilon_0,\epsilon_1]\cdot[x_0,\dots,x_{m+n+1}] \mapsto \big[\epsilon_0x_0,\dots,\epsilon_0x_m,\,\epsilon_1x_{m+1},\dots,\epsilon_1x_{m+n+1}\big]
\end{align*}
The action fixes the images of the noninvertible points $[1,0]$ and $[0,1]$ of $\PP^1$:
\begin{align*}
  \PP^m &= [1,0]\cdot \mrm \PP^{m+n+1} = \{[x_0,\dots,x_m,0,\dots,0]\} \\
  \PP^n &= [0,1]\cdot \mrm \PP^{m+n+1} = \{[0,\dots,0,x_{m+1},\dots,x_{m+n+1}]\}
\end{align*}
and flows the rest of the points along the join lines.

The measure of $\PP^{m+n+1}$ is the join of the measures of $\PP^m$ and $\PP^n$ in the sense that:
\begin{align*}
  \mu_{\PP^{m+n+1}}(U\star V) = \mu_{\PP^m}(U)\cdot\mu_{\PP^n}(V)
\end{align*}

\begin{lemma}
  Away from $\PP^m\cup\PP^n$:
  \begin{align*}
    \alpha(X\times Y) = \alpha(X)\star \alpha(Y)
  \end{align*}
\end{lemma}

\begin{corollary}
  \begin{align*}
    \kappa_{X\times Y}(U\times V) = \int_{\PP^{m+n+1}} \sum_{U\times V}\alpha(X\times Y)
    = \int_{\PP^n}\sum_U\alpha(X) \cdot \int_{\PP^m}\sum_V\alpha(Y) = \kappa_X(U)\cdot\kappa_Y(V)
  \end{align*}
\end{corollary}

The author believes this can be dressed up to prove Proposition~\ref{prop:strat-sub}, given the Br\"ocker-Kuppe description of curvature as the average stratified Morse index of projections in all directions: shrinking just distorts the sense of ``average'' according to the above flow.

\pagebreak[3] %

\begin{example}  \mbox{}
  \begin{center}
    \includegraphics[width=118pt]{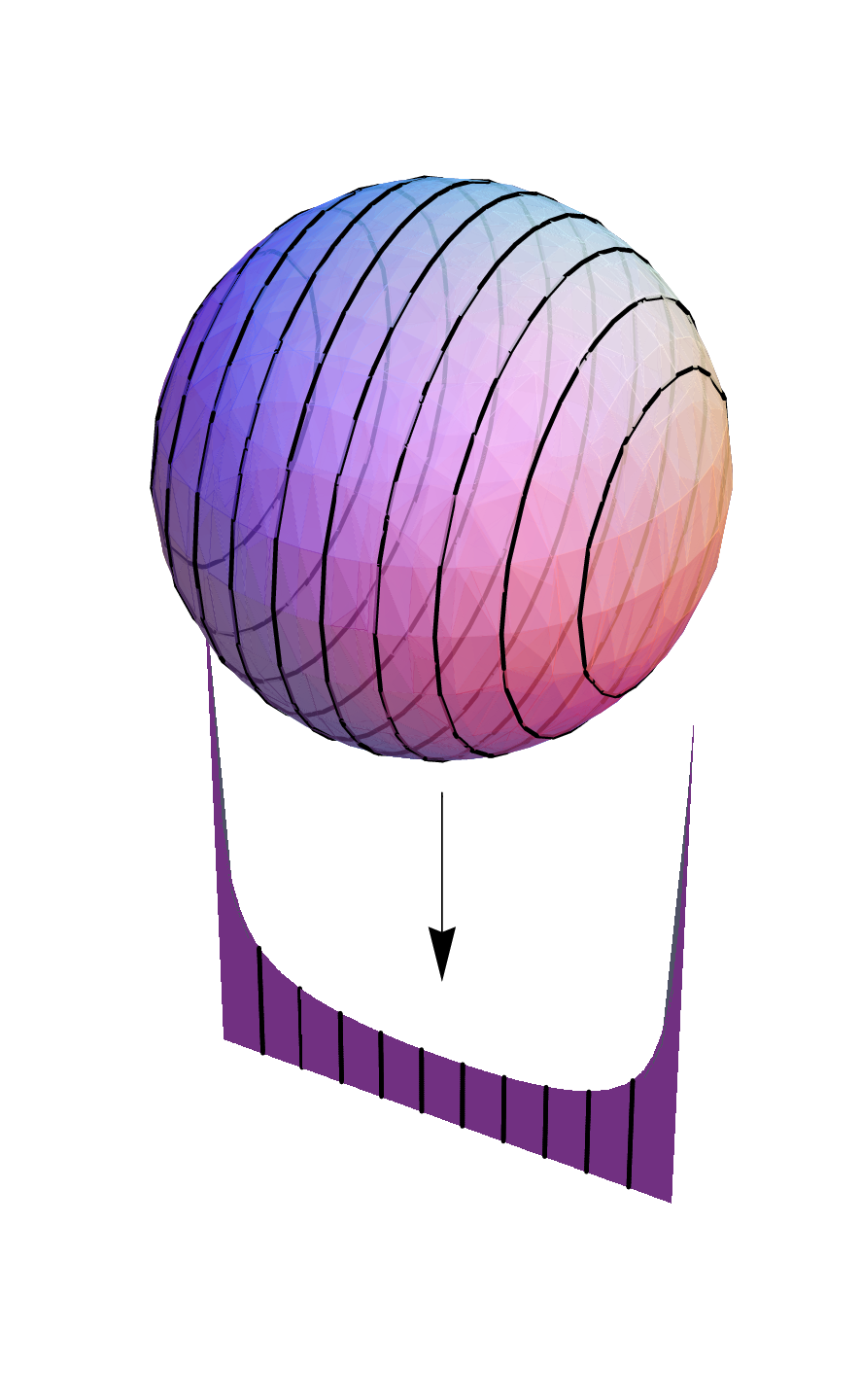}
    \includegraphics[width=118pt]{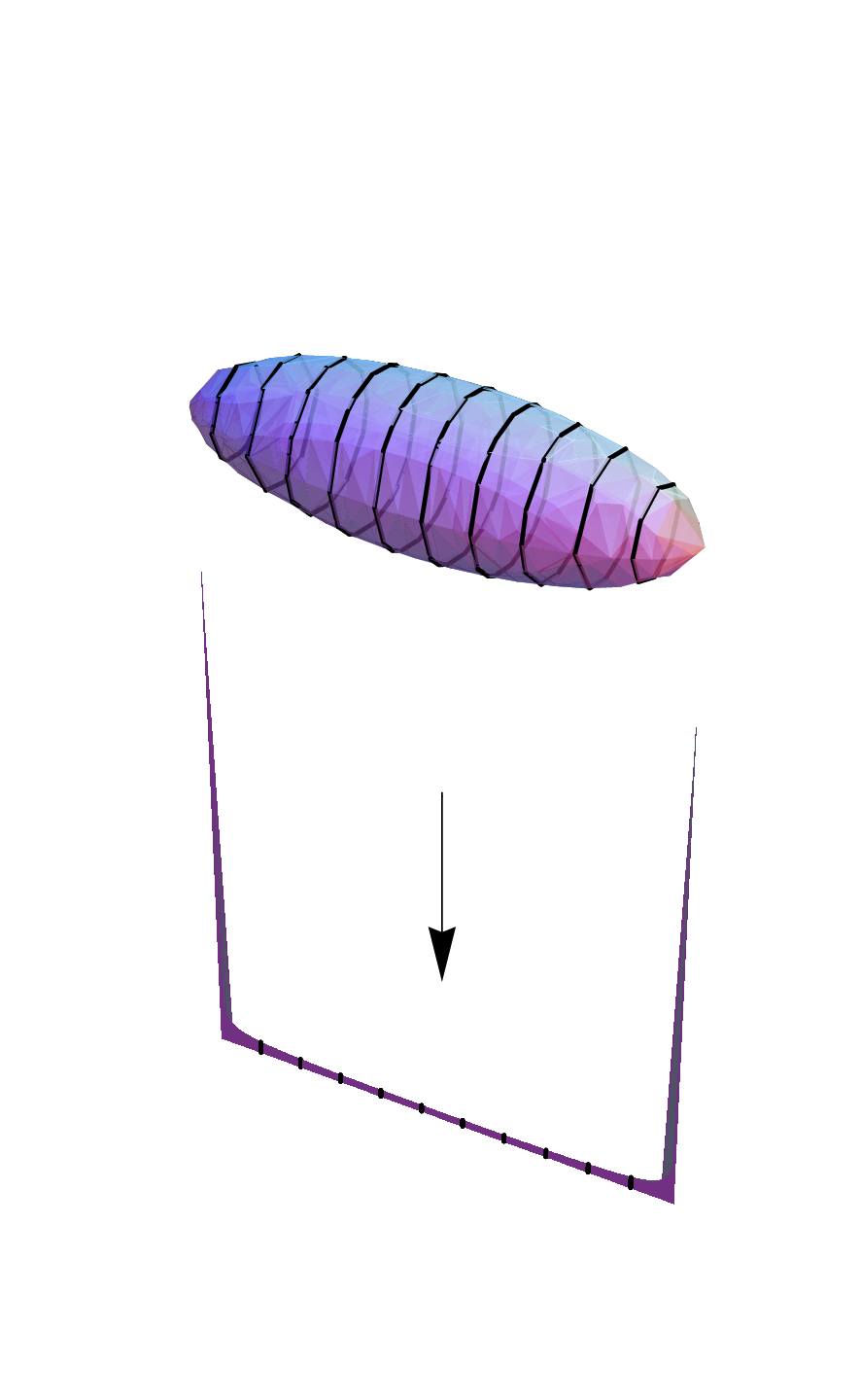}
    \includegraphics[width=118pt]{euler-graph-5.png}
  \end{center}
\end{example}

\appendix

\section{Invariance Under Barycentric Subdivision}

\label{apx:bary}

This section contains a combinatorial proof of Theorem~\ref{thm:crumple}. Its geometric meaning remains mysterious to the author. (Subdividing a simplicial complex and giving each resulting 1-simplex length~1 essentially crumples up the complex, albeit in a symmetric way.)

\begin{lemma}
  If $\Delta^{(1)}$ is the first barycentric subdivision of an $n$-simplex $\Delta$ then:
  \begin{align*}
    \int_{\Delta^{(1)}} \alpha\,\dchi = \alpha(\hat{\Delta})
  \end{align*}
\end{lemma}

\begin{proof}
  One can describe $\Delta^{(1)}$ as the nerve of the category whose objects are nonempty subsets of the vertex set of $\Delta$ and whose morphisms are strict inclusions. In particular, each $i$-simplex of $\Delta^{(1)}$ corresponds to a length-$(i+1)$ chain of strict inclusions:
  \begin{align*}
    A_0 \to A_1 \to \cdots \to A_i
  \end{align*}
  Intuitively, the vertices of the corresponding $i$-simplex are the barycenters of each set $A_k$.

  Associated to each such $i$-simplex is the sequence of integers:
  \begin{align*}
    \Big( \big|A_0\big|, \big|A_1\setminus A_0\big|, \big|A_2\setminus A_1\big|, \dots, \big|A_i\setminus A_{i-1}\big| \Big)
  \end{align*}
  called the \emph{signature} $(s_0,\dots,s_i)$ of the simplex.

  \medskip

  To prove the lemma we will write:
  \begin{align*}
    \sum_{\sigma \in \Delta^{(1)}} (-1)^{\dim(\sigma)} \hat{\sigma} 
    = \sum B(s_0,\dots,s_i)
  \end{align*}
  where $B(s_0,\dots,s_i)$ consists of the terms corresponding to simplices with signature $(s_0,\dots,s_i)$ and then show that if $i>0$ and $s_0+\cdots+s_i = n+1$ then:
  \begin{align*}
    B(s_0,\dots,s_i)+B(s_0,\dots,s_{i-1})=0
  \end{align*}
  The lemma will then follow since the only simplex of $\Delta^{(1)}$ whose signature does not belong to such a cancelling pair is the barycenter $\hat{\Delta}$, which has signature $(n+1)$.

  \smallskip

  To simplify notation, assume from now on that $s_0+\cdots+s_i=n+1$.

  \smallskip

  The action of the symmetric group $\mrm{S}_{n+1}$ on the vertex set of $\Delta$ induces a transitive action on the set of simplices of $\Delta^{(1)}$ with a given signature $(s_0,\dots,s_i)$.

  The stabilizer of any such simplex is isomorphic to the product of symmetric groups:
  \begin{align*}
    \mrm{S}_{s_0} \times \cdots \times \mrm{S}_{s_i}
  \end{align*}

  By the orbit stabilizer theorem, then, the number of simplices of $\Delta^{(1)}$ with signature $(s_0,\dots,s_n)$ equals the multinomial coefficient:
  \begin{align*}
    \binom{n+1}{s_0,\dots,s_n} = \frac{(n+1)!}{s_0! \cdots s_n!}
  \end{align*}
 
  Now, $B(s_0,\dots,s_i)$ is symmetric in the vertices of $\Delta$, so we may concentrate on the coefficient of a single vertex $v$. A simple argument shows that the number of simplices of $\Delta^{(1)}$ with signature $(s_0,\dots,s_i)$ in which the chosen vertex $v$ appears in the set $A_k$ but not in $A_{k-1}$ equals:
  \begin{align*}
    \binom{n}{s_0,\dots,s_k-1,\dots,s_i}
    = \frac{s_k}{n+1} \binom{n+1}{s_0,\dots,s_i}
  \end{align*}
  The coefficient of $v$ in the barycenter of such a simplex equals:
  \begin{align*}
    \frac1{i+1} \sum_{j=k}^i \frac1{\sum_{l=0}^j s_j}
  \end{align*}
  This is because the vertices of such a simplex are the barycenters of the sets $A_l$, and the vertex $v$ first appears in $A_k$.
  
  Summing over $k$ and multiplying by $(-1)^{i-1}$ we obtain that the coefficient of $v$ in $B(s_0,\dots,s_i)$ equals:
  \begin{align*}
    {(-1)^{i-1}} \frac1{i+1} &\sum_{k=0}^i \left( \sum_{j=k}^i \frac1{\sum_{l=0}^j s_j} \right) \frac{s_k}{n+1} \binom{n+1}{s_0,\dots,s_i} \\
    = \frac{(-1)^{i-1}}{n+1} \binom{n+1}{s_0,\dots,s_i} \frac1{i+1} &\sum_{k=0}^i \left( \sum_{j=k}^i \frac1{\sum_{l=0}^j s_j} \right) s_k
  \end{align*}
  The nested sum simplifies considerably since:
  \begin{equation*}
    \begin{split}
      \Big(\frac1{s_0} + \cdots + \frac1{s_0+\cdots+s_n}\Big)s_0
      &+ \Big(\frac1{s_0+s_1} + \cdots + \frac1{s_0+\cdots+s_n}\Big)s_1 \\
      &+ \Big(\frac1{s_0+s_1+s_2} + \cdots + \frac1{s_0+\cdots+s_n}\Big)s_2 \\
      & \,\,\, \vdots \\
      &+ \Big( \frac1{s_0+\cdots+s_n} \Big) s_n
    \end{split}
    \end{equation*}
    \begin{align*}
      &
      \begin{aligned}
        =\frac1{s_0} s_0 + \frac1{s_0+s_1}(s_0+s_1) 
        &+ \frac1{s_0+s_1+s_2}(s_0+s_1+s_2) \\ + \cdots &+ \frac1{s_0+\cdots+s_i}(s_0+\cdots+s_i) \\
      \end{aligned} \\
      &= i+1
  \end{align*}
  Therefore:
  \begin{align*}
    B(s_0,\dots,s_i) &= (n+1) \hat{\Delta} \cdot \frac{(-1)^{i-1}}{n+1} \binom{n+1}{s_0,\dots,s_i} \frac1{i+1} \cdot (i+1) \\
    &= (n+1) \hat{\Delta} \cdot \frac{(-1)^{i-1}}{n+1} \binom{n+1}{s_0,\dots,s_i}
  \end{align*}

\smallskip

Virtually the same argument shows that if $i>0$ then:
\begin{align*}
  B(s_0,\dots,s_{i-1}) &= \frac{(-1)^{i-2}}{n+1} \binom{n+1}{s_0,\dots,s_i} \\
  &= - B(s_0,\dots,s_i)
\end{align*}
  The main point here is that the stabilizer under the action of $\mrm{S}_{n+1}$ of a simplex of $\Delta^{(1)}$ with signature $(s_0,\dots,s_{i-1})$ is isomorphic to the stabilizer of a simplex with signature $(s_0,\dots,s_i)$.
\end{proof}

\begin{corollary}
  \label{cor:barycentric}
  For any $i\ge1$:
  \begin{align*}
    \int_X \alpha \; \dchi = \int_{X^{(i)}} \alpha^{(i)} \; \dchi
  \end{align*}
  where $\alpha^{(i)} : X^{(i)}\to\RR^{(i)}$ is the linear extension of $\alpha$ to the $i$th barycentric subdivision $X^{(i)}$ of $X$.
\end{corollary}

\section*{Acknowledgments}

Thanks to Baryshnikov and Ghrist for their inspiring papers. Thanks to Robert MacPherson, Jack Morava, Burt Totaro, Andrew Ranicki, Hermann Karcher, Vin de Silva and J\"org Sch\"urmann for encouraging conversations. Thanks to Jacek Brodzki and the EPSRC for financial support from 2011--2. Thanks to the organizers of \href{http://www.icms.org.uk/workshops/atmcs5}{ATMCS 5} for the opportunity to speak about an early version of this work in 2012.

\def\cprime{$'$}

\end{document}